%% file: PartialExchangeability_180116.tex
\documentclass[11pt]{article}

\usepackage{subfigure}
\usepackage{tikz}
\usepackage{graphicx}
\usetikzlibrary{arrows}
\input{packages}
\input{macros}

\title{de Finetti reductions for partially exchangeable probability distributions}
\author{Ivan Bardet}
    \affil[1]{\small Institut des Hautes \'{E}tudes Scientifiques, Universit\'{e} Paris-Saclay, 91440 Bures-sur-Yvette, France}
    \author{C\'ecilia Lancien}  
    \affil[2]{\small Departamento de An\'{a}lisis Matem\'{a}tico, Universidad Complutense de Madrid, 28040 Madrid, Spain \& Instituto de Ciencias Matem\'{a}ticas, 28049 Madrid, Spain}
    \author{Ion Nechita}
	\affil[3]{\small Laboratoire de Physique Th\'{e}orique, Universit\'{e} de Toulouse, F-31062 Toulouse, France}

\begin{document}

\maketitle

\begin{abstract}
We introduce a general framework for de Finetti reduction results, applicable to various notions of partially exchangeable probability distributions. Explicit statements are derived for the cases of exchangeability, Markov exchangeability, and some generalizations of these. Our techniques are combinatorial and rely on the ``BEST'' theorem, enumerating the Eulerian cycles of a multigraph. 
\end{abstract}

\tableofcontents

\section{Introduction}\label{sect:intro}

The general motivation behind all de Finetti type results could be phrased as follows: how to make use of the invariances that a probability distribution in many variables might have to reduce its study to that of simpler probability distributions? What we mean by ``invariances'' of a sequence of random variables is that exchanging some of the variables in it does not change its distribution. The first example that one could think of is clearly when all variables play the same role. Assuming for concreteness that we are dealing with sequences of binary random variables, we see that, in this fully exchangeable case, the probability of a given $n$-variable sequence is entirely determined by its number of $0$'s and $1$'s, as if the $n$ variables had been drawn independently according to the same distribution. The goal of de Finetti type statements would thus be to make this observed similarity between the two random models mathematically precise. In this paper, we are also exploring less restrictive notions of exchangeability, and for each of them trying to relate in a quantitative way any multi-variable sequence exhibiting this symmetry to simpler sequences.

\subsection{General setting}
Let us now introduce more concretely the mathematical setting in which we will be working. Let $V$ be a finite set of cardinality $d$. We are interested in probability distributions on $V^n$, for some positive integer $n$, which satisfy specific kinds of symmetries. The notion of \emph{partial exchangeability} was introduced by de Finetti \cite{dF} in order to address distributional symmetries. Denote by $\Pcal(V^n)$ the convex set of probability distributions on $V^n = V \times \cdots \times V$ ($n$ times). Given an equivalence relation $\sim$ on $V^n$, a probability distribution $P\in\Pcal(V^n)$ is called \emph{partially exchangeable} for the relation $\sim$ (or $\sim$--exchangeable) if, for any $v,w\in V^n$, $P(v)=P(w)$ whenever $v\sim w$. We denote by $\Pcal_\sim(V^n)$ the corresponding convex subset of $\Pcal(V^n)$. 

Maybe the most well-known examples of partial exchangeability are the two following ones:
\begin{enumerate}
\item \emph{Exchangeability}: $P\in\Pcal(V^n)$ is exchangeable if it assigns the same probability to any two sequences in $V^n$ that have the same number of $i$ for each $i\in V$.
\item \emph{Markov exchangeability}: $P\in\Pcal(V^n)$ is Markov exchangeable if it assigns the same probability to any two sequences in $V^n$ that have the same number of transitions from $i$ to $j$ for each $(i,j)\in V$.
\end{enumerate}

In this work, we also introduce a generalization of the two notions above, which we call \emph{$\ell$-Markov exchangeability}: $P\in\Pcal(V^n)$ is $\ell$-Markov exchangeable if it assigns the same probability to any two sequences in $V^n$ that have the same number of $\ell$-transitions $i_1\to i_2 \to \cdots\to i_{\ell}$ for each $(i_1,...,i_{\ell})\in V$.

We shall additionally discuss notions of double partial exchangeability, such as \emph{double exchangeability}: if $V=V_1\times V_2$ for two finite sets $V_1$ and $V_2$, $P\in\Pcal(V^n)$ is doubly exchangeable if it assigns the same probability to any two sequences in $V^n$ that have the same numbers of $i$ and $j$ for each $i\in V_1$ and $j\in V_2$, or otherwise said if the marginals of $P$ over $V_1^n$ and $V_2^n$ are both exchangeable. Of course, the notion above generalizes to \emph{double Markov exchangeability}, or even \emph{double $\ell$-Markov exchangeability}.

These are all the examples that we cover in this article, but one may think of many others. For instance, one can consider the case where $V$ is the Cartesian product of more than two parties. One can also partition $V$ and impose either exchangeability or Markov exchangeability on each part; we shall prove general results which allow one to mix and match the different notions of partial exchangeability to fit a specific symmetry model. 

\begin{figure}[ht]
\begin{center}
\includegraphics[width=.45\textwidth]{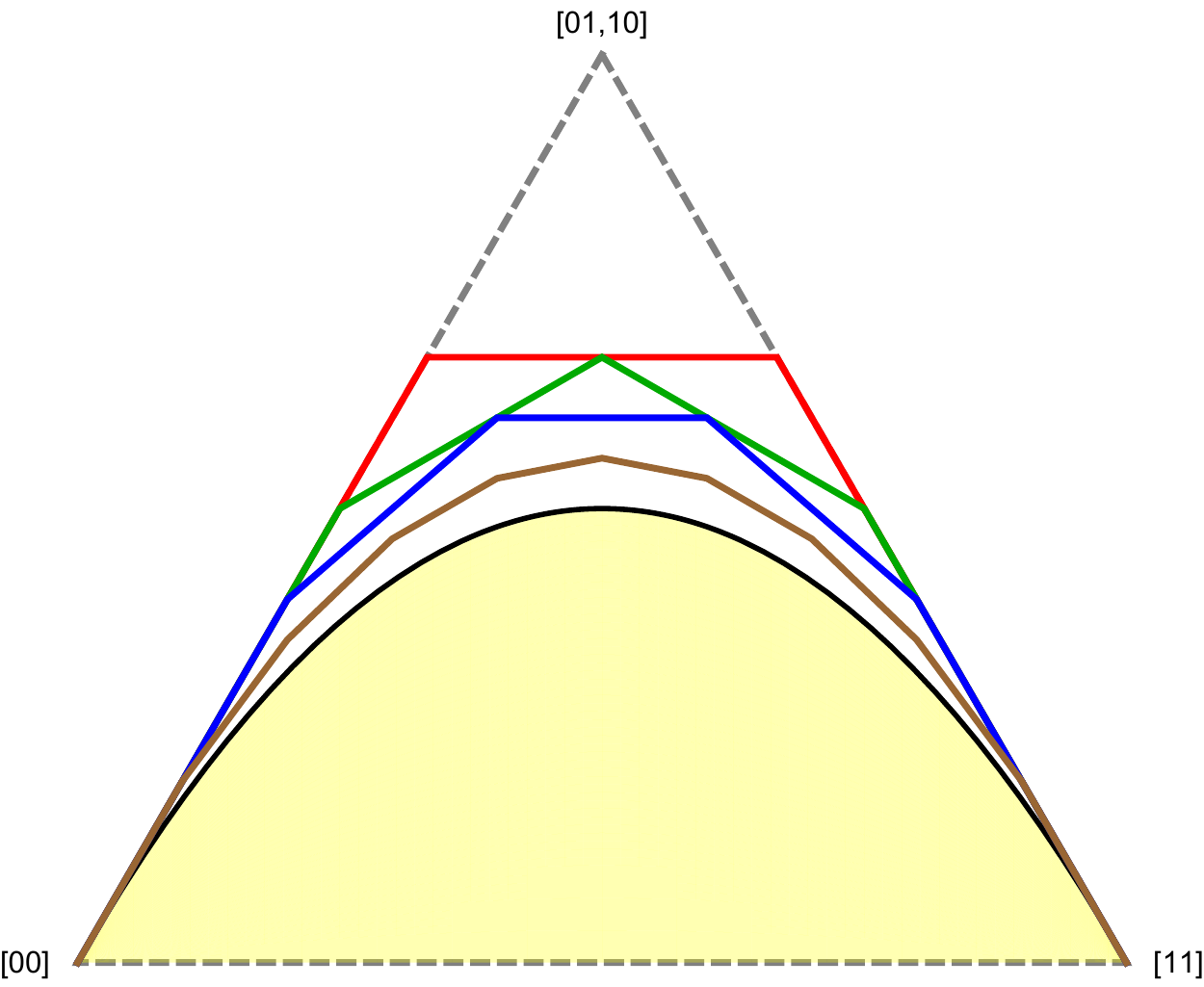}\qquad
\includegraphics[width=.45\textwidth]{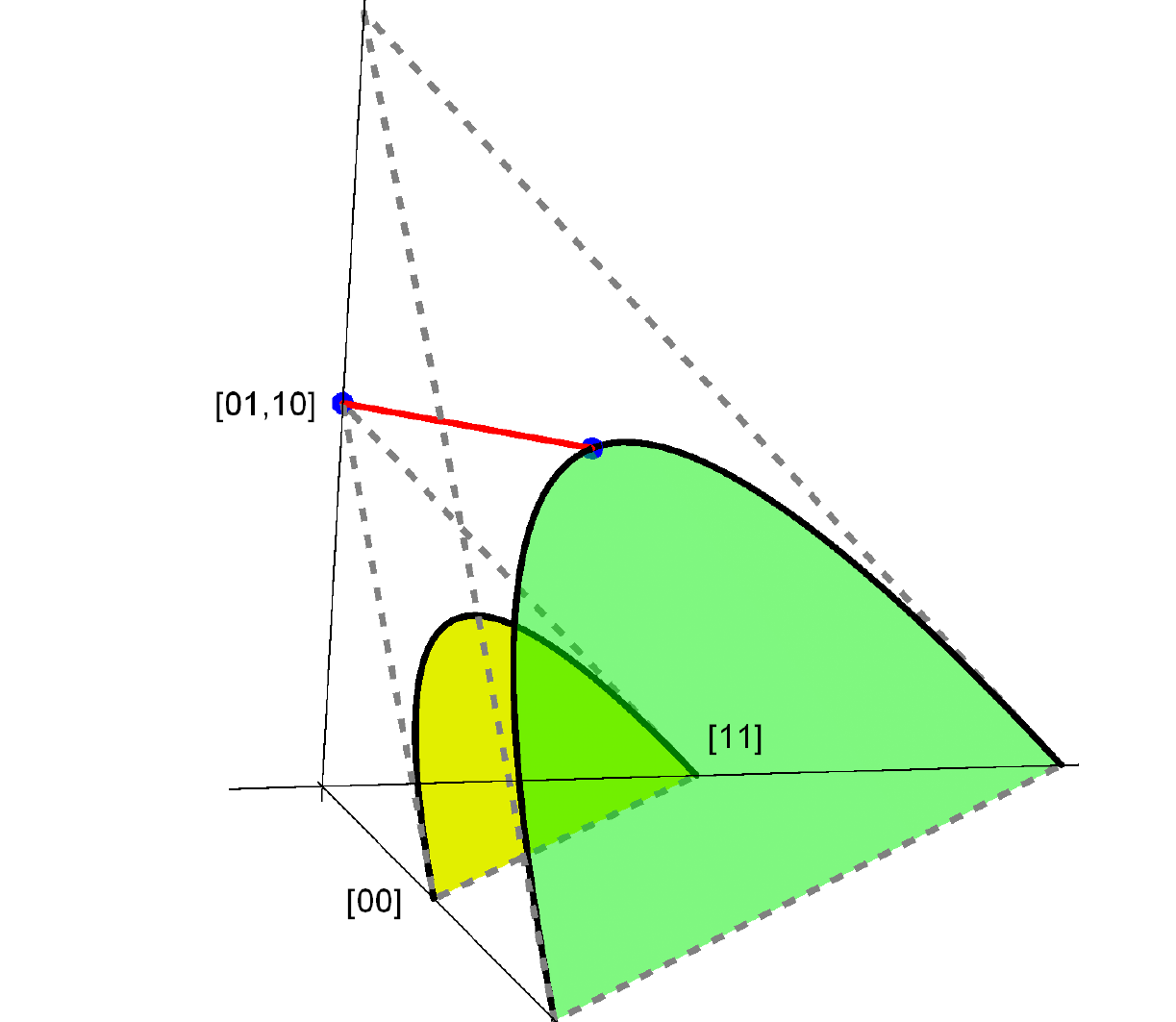}
\end{center}
\caption{Exchangeability for $V=\{0,1\}$ and $m=2$. In the left panel, the three extreme distributions are the nodes of the triangle, the i.i.d.~distributions are represented by the black curve, while the mixture of i.i.d.~distributions are in the filled yellow area. The coloured polygonal lines represent $m=2$-marginals of exchangeable distributions on $V^n$, with $n=\textcolor{red}{3},\textcolor{green}{4},\textcolor{blue}{5},\textcolor{brown}{10}$ respectively. In the right panel, the setting of de Finetti reductions: the $([00],[11],[01,10])$-simplex is (coordinate-wise) smaller than the set (green filled area) of mixtures of i.i.d.~distributions, scaled by a factor of 2; indeed, the vertex corresponding to $[01,10]$ has the same $z$-coordinate as the top-point of the green-filled area in the foreground.}
\end{figure}

\paragraph{}For a general equivalence relation on $V^n$, the finite de Finetti Theorem \cite{dF} characterizes the extreme points of $\Pcal_\sim(V^n)$.

\begin{theo}\label{theo_dFextremepoints}
Given an equivalence relation $\sim$ on $V^n$, $\Pcal_\sim(V^n)$ is a simplex whose extreme points are the uniform distributions on the equivalence classes for $\sim$.
\end{theo}

For exchangeability and Markov exchangeability, the uniform distributions on the equivalence classes for $\sim$ can be quantitatively compared to much simpler distributions (via urn models for instance), namely: independently and identically distributed (i.i.d.) probability distributions in the exchangeable case and Markovian probability distributions in the Markov exchangeable case. Combining these comparison results with the above characterization of the extreme points of $\Pcal_\sim(V^n)$ as the uniform distributions on the equivalence classes for $\sim$, then allows to derive, in these two cases, the so-called \emph{de Finetti theorems}. Their infinite versions read as follows: If $P$ is an infinitely exchangeable (resp.~Markov exchangeable) probability distribution, then for any positive integer $m$, its marginal on $V^m$ is a convex combination of i.i.d. (resp.~Markovian) probability distributions. The exchangeable case was proven by de Finetti in his seminal paper \cite{dF} while the Markov exchangeable case was proven by Diaconis and Freedman in \cite{DiaFreed80}. In practice, the finite (approximate) versions of these statements are often more useful. They were established later, by Diaconis and Freedman for exchangeability \cite{diaconis1980finite} and by Zaman for Markov exchangeability \cite{Zaman1986finite}. Informally, they tell us the following: If $P$ is an exchangeable (resp.~Markov exchangeable) probability distribution on $V^n$, then for any positive integer $m\ll n$, its marginal on $V^m$ is well-approximated (say in $\ell^1$-distance) by a convex combination of i.i.d. ( resp.~Markovian) probability distributions. 

More recently, motivated by questions in theoretical computer science and information theory, people got interested in another sort of de Finetti type statements, often referred to as \emph{de Finetti reductions}. The goal here is slightly different: given a partially exchangeable probability distribution $P\in\Pcal_\sim(V^n)$, one is not interested in \emph{approximating} the marginals of $P$ by simpler probability distributions, but just in \emph{upper bounding} $P$ itself by simpler probability distributions. A result of this kind has been established in \cite{lancien2016flexible} in the case of exchangeability: If $P$ is an exchangeable probability distribution on $V^n$, then it is upper bounded (point-wise) by a convex combination of i.i.d.~probability distributions times a pre-factor which is \emph{polynomial} in $n$. The fact that this pre-factor can be shown to be polynomial in $n$, and not exponential, is a crucial point in the applications of this result. Let us try to briefly explain why, by describing in a very vague way the general scenario that we have mind, which is the one of multiplicativity problems. Imagine that you are given a convex order-preserving functional $\omega$ which is known to be upper bounded by $1-\delta$ on a given subset $\mathcal{P}_1$ of $\Pcal(V)$, and you would like to upper bound $\omega^{\otimes n}$ on a related subset $\mathcal{P}_n$ of $\Pcal(V^n)$. If $\mathcal{P}_n\subset\mathcal{P}_1^{\otimes n}$, then clearly $\omega^{\otimes n}$ is upper bounded by $(1-\delta)^n$ on $\mathcal{P}_n$, which goes to $0$ exponentially with $n$. If $\mathcal{P}_n$ is more general than this, but still such that any of its elements can be upper bounded by an element of $\mathcal{P}_1^{\otimes n}$ times a $\mathrm{poly}(n)$ pre-factor, then again $\omega^{\otimes n}$ is upper bounded by a quantity which goes to $0$ exponentially with $n$. 

In this paper, we describe a unified approach to derive de Finetti reductions for various notions of partial exchangeability. More precisely, this means that, for various equivalence relations $\sim$ on $V^n$, we are able to upper bound any element of $\Pcal_\sim(V^n)$ by a convex combination of much simpler probability distributions times a pre-factor which is polynomial in $n$.

\subsection{Summary of our results}

Before going into greater details, let us specify here once and for all some notation that we will be using repeatedly throughout the whole paper. Let $\sim$ be an equivalence relation on $V^n$. We let $N$ be the number of equivalence classes for this relation and we write $\Ccal_1,...,\Ccal_N$ for the $\sim$ equivalence classes; we also denote by $Q_1,...,Q_N$ the uniform distributions on each of these classes, that is, the extreme points of the simplex $\Pcal_\sim(V^n)$. Given $P,Q$ two probability distributions on the same (finite) set $Z$ and given $C\ge 1$, we write $P\le C\,Q$ if the point-wise inequality $P(z)\le C\,Q(z)$ for each $z\in Z$ holds. Finally, we denote by $F(P,Q)$ the fidelity between $P$ and $Q$, i.e.~$F(P,Q)=\sum_{z\in Z}\sqrt{P(z)}\sqrt{Q(z)}$.

\paragraph{}For each of the previously mentioned notions of partial exchangeability, defined by an equivalence relation $\sim$ on $V^n$, we shall prove two types of de Finetti reductions. 

The first result is a flexible form of de Finetti reduction, which was already obtained in \cite{lancien2016flexible} in the case of exchangeability. The second result is a universal form of de Finetti reduction for conditional probability distributions, which was already proved in \cite{A-FR15} in the case of exchangeability. However, even in the latter case, our proof is much simpler and thus worthwhile. Let us introduce formally their general framework. We recall that, given an equivalence relation $\sim$ on $V^n$, we write $P\in\mathcal P_\sim(V^n)$ if $P(v)=P(w)$ whenever $v\sim w$. In the case where $V=A\times X$ is a product space, $P\in\mathcal P_\sim(A^n\times X^n)$ thus simply means that $P(a,x)=P(b,y)$ whenever $(a,x)\sim (b,y)$.

\begin{defi}\label{def:deF-reduction}
	Let $V$ be a finite alphabet. Let $\sim$ be a sequence of equivalence relations on $\mathcal P(V^n)$ and consider a sequence of distinguished subsets of $\sim$--exchangeable probability distributions $\Pi(V^n) \subseteq \mathcal P_\sim(V^n)$. We say that the pair $(\sim, \Pi)$ admits a \emph{flexible de Finetti reduction} if, for any probability distribution $P \in \mathcal P_\sim(V^n)$, we have
\begin{equation}\label{eq:def-deF-reduction}
P \leq \mathrm{poly}(n) \int_{\pi \in \Pi(V^n)} F(P, \pi)^2 \pi \, d\nu(\pi),
\end{equation}
where $\mathrm{poly}(n)$ is a polynomial in $n$ and $\nu$ is a probability distribution on $\Pi(V^n)$.
\end{defi}

\begin{defi}\label{def:deF-reduction-cond}
	Let $A,X$ be finite alphabets. Let $\sim$ be a sequence of equivalence relations on $\mathcal P(A^n\times X^n)$ and consider a sequence of distinguished subsets of $\sim$--exchangeable conditional probability distributions $\Pi(A^n\times X^n) \subseteq \mathcal P_\sim(A^n\times X^n)$. We say that the pair $(\sim, \Pi)$ admits a \emph{universal conditional de Finetti reduction} if, for any probability distribution $P\in\Pcal_\sim(A^n\times X^n)$, we have
	\begin{equation}\label{eq:theo_conditionaldFreduction}
	P_{A^n|X^n}\le \mathrm{poly}(n)\, \int_{\pi\in \Pi(A^n\times X^n)}\pi_{A^n|X^n}\,d\nu(\pi)
	\end{equation}
	where $\mathrm{poly}(n)$ is a polynomial in $n$ and $\nu$ is a probability distribution on $\Pi(A^n\times X^n)$.
\end{defi}

Note a major difference between the two definitions above: the upper bounding probability distribution in \eqref{eq:def-deF-reduction} depends on the considered $P \in \mathcal P_\sim(V^n)$ (it is \emph{flexible}) while the upper bounding conditional probability distribution in \eqref{eq:theo_conditionaldFreduction} is the same for all $P\in\Pcal_\sim(A^n\times X^n)$ (it is \emph{universal}). 

\paragraph{}We are now in position to write down informally our two main results. The precise statements appear as Corollaries \ref{cor:exchangeable}, \ref{cor:Markov-exchangeable}, \ref{cor:lMarkov-exchangeable} (for the probability distribution version) and as Corollary \ref{cor:exchangeable-cond} (for the conditional probability distribution version). 

\begin{theo}\label{theo_dFreduction}
The following notions of partial exchangeability admit a flexible de Finetti reduction:
\begin{enumerate}
	\item Exchangeable distributions, together with $\Pi = \{\text{i.i.d.~distributions}\}$
	\item Markov exchangeable distributions, together with $\Pi = \{\text{Markov distributions}\}$.
	\item $\ell$-Markov exchangeable distributions, a natural generalization of the notion of Markov exchangeability.
\end{enumerate}
Explicit bounds on the degree of the polynomial in \eqref{eq:def-deF-reduction} can be given in every situation, as functions of the size of the alphabet $V$ (and of $\ell$).
\end{theo}

\begin{theo}\label{theo_dFreduction-cond}
Exchangeable distributions, together with $\Pi = \{\text{i.i.d.~distributions}\}$, admit a universal conditional de Finetti reduction.
An explicit bound on the degree of the polynomial in \eqref{eq:theo_conditionaldFreduction} can be given as a function of the size of the alphabets $A,X$.
\end{theo}

 As it shall be clear from the proofs, for all the notions of partial exchangeability that we have mentioned, the integrals in equations \eqref{eq:def-deF-reduction} and \eqref{eq:theo_conditionaldFreduction} can in fact be replaced by finite sums (see Subsection \ref{sec:general-statements} for the precise expressions). 

\paragraph{}Our method of proof is quite general and could thus, at least in principle, be applied to other notions of partial exchangeability. Our starting point is a strengthening of Theorem \ref{theo_dFextremepoints}, enlightening the role of the fidelity, which appears as Theorem \ref{theo_dFextremepoints}. Intuitively, what the latter tells us is that it is enough to consider the reduction problem for the extreme points of the simplex $\Pcal_\sim(V^n)$, i.e.~for the uniform distributions on the equivalence classes. Then, combinatorial arguments allow us to compare each of these extreme points with a simpler specific point of the convex subset of interest, subsequently concluding. This approach was already successful in order to prove finite de Finetti Theorems in the case of exchangeability \cite{diaconis1980finite} and Markov exchangeability \cite{Zaman1986finite}. Interestingly for us, the latter work (following \cite{Zaman1984urn})) emphasizes the role of the BEST Theorem in estimating the cardinal of the equivalence classes, and we are able to pursue a similar argument in the case of $\ell$-Markov exchangeability. As for the cases of double exchangeability and double exchangeability, they can be directly deduced from the previous cases.

\paragraph{}This article is structured as follows: Section \ref{sect:uniform} establishes flexible de Finetti reductions for the different notions of partial exchangeability that we presented, while Section \ref{sect:conditional} establishes their universal conditional counterparts. In both sections, we first highlight a general method to derive them, and then make key use of combinatorial statements about the extremal distributions of each set of partially exchangeable distributions. Finally in Section \ref{sect:NLgames}, we discuss possible applications of our results, to study the repetition of multi-player non-local games. 

\section{Flexible de Finetti reductions for partially exchangeable probability distributions} \label{sect:uniform}

Our goal in this section is to show that several important notions of partial exchangeability admit a flexible de Finetti reduction, in the sense of Definition \ref{def:deF-reduction}. For that, we start with stating a very general result, and then apply it to many particular cases of interest. 

\subsection{General statement}\label{sec:general-statements}

We begin with proving a de Finetti reduction for partially exchangeable probability distributions which is a priori applicable to any notion of partial exchangeability.

\begin{lem}\label{lem_constrainedDeFinettiTheorem}
	Suppose that there exist functionals $\alpha_1(n),...,\alpha_N(n)$ and probability distributions $\pi_k\in \Pcal_\sim(V^n)$, $k=1,...,N$ such that
	\begin{equation}\label{eq1_lem_constrainedDeFinettiTheorem}
	Q_k\le\alpha_k(n)\,\pi_k\qquad\text{ for all }k=1,...,N.
	\end{equation}
	Then, for any probability distribution $P\in\Pcal_\sim(V^n)$, one has
	\begin{equation}\label{eq2_lem_constrainedDeFinettiTheorem}
	P\le N\times\alpha(n)^2\sum_{k=1}^N\, \frac{1}{N}F(P,\pi_k)^2\,\pi_k
	\end{equation}
	where $\alpha(n)=\max\,\{\alpha_k(n),\ k=1,...,N\}$. 
	
	In particular, if $N$ and $\alpha(n)$ is polynomial in $n$, then the pair $(\sim, \{\pi_1,\ldots,\pi_N\})$, admits a flexible de Finetti reduction in the sense of Definition \ref{def:deF-reduction}, with $\mathrm{poly}(n)=N\times\alpha(n)^2$.
\end{lem}

\begin{proof}
	Let us begin with stating a few useful facts:
	\begin{enumerate}
		\item For any $P\in\Pcal_\sim(V^n)$, there exists a unique decomposition
		\begin{equation}\label{eq1_proof_lem__constrainedDeFinettiTheorem}
		P=\sum_{k=1}^N{\mu_k(P)Q_k}.
		\end{equation}
		The existence of the decomposition comes from the fact that the $Q_k$ are the extreme points of $\Pcal_\sim(V^n)$ and the Krein-Milman Theorem. Unicity is clear since the $Q_k$ have disjoint supports as probability distributions. In other words, the set $\mathcal P_\sim(V^n)$ is a simplex, having $Q_k$ as extremal points. 
		\item By a simple computation and using the fact that any $P$ in $\Pcal_\sim(V^n)$ is constant on $\Ccal_k$, one obtains that, for each $k=1,...,N$,
		\[\mu_k(P)=F(P,Q_k)^2.\]
		\item By inequality \eqref{eq1_lem_constrainedDeFinettiTheorem}, one gets that, for each $k=1,...,N$, 
		\[ F(P,Q_k)^2\le \alpha_k(n) F(P,\pi_k)^2. \]
	\end{enumerate}
	Putting those three facts together yields inequality \eqref{eq2_lem_constrainedDeFinettiTheorem}.
\end{proof}

\begin{remark}
	Lemma \ref{lem_constrainedDeFinettiTheorem} can be interpreted nicely in the framework of the modular theory of von Neumann algebras \cite{Tak}. Indeed, if we associate to each probability distribution $P$ on $V^n$ its ``modular'' density $P^{1/2}$, viewed as an element of $l^2(V^n)$, then the fidelity is exactly the scalar product on $l^2(V^n)$ restricted to $\Pcal(V^n)$. We remark furthermore that the $\sqrt{Q_k}$'s are pairwise orthogonal for this scalar product and, since they are probability distributions, have norm $1$. Consequently, and because of equation \eqref{eq1_proof_lem__constrainedDeFinettiTheorem}, they form an orthonormal basis of the space spanned by $\Pcal_\sim(V^n)$ and $\mu_k(P)=F(P,Q_k)^2$. This also gives an alternative proof of Theorem \ref{theo_dFextremepoints}.
\end{remark}

\subsection{The case of exchangeability}\label{sect:exch}

In the case of exchangeability, each equivalence class $\Ccal_t$ can be characterized by its \emph{type}, that is by the $d$-tuple $t=(t_1,...,t_d)$ where for each $i=1,...,d$, $t_i$ is the number of $i$ in a sequence $v=(v_1,...,v_n)\in\Ccal_k$. More precisely, $t_i$ is a non-negative integer for all $i=1,...,d$ and $t_1+\cdots+t_d=n$. Note that the number $N$ of equivalence classes is exactly the number of such $d$-tuples, i.e.
\[N=\binom{n+d-1}{n}.\]
Crucially for our purposes, note that $N$ is polynomial in $n$, namely $N\leq (n+1)^{d-1}$.

\begin{lem}\label{lem_exchangeability}
For all types $t = (t_1, \ldots, t_d)$, we have
\begin{equation}\label{eq1_lem_exchangeability}
Q_t\le\alpha(n)\,\pi_t^{\otimes n},
\end{equation}
where $\pi_t$ is the probability distribution on $V$ such that $\pi_t(i)=t_i/n$ for all $i=1,...,d$, and where $\alpha(n)$ is given by
\begin{equation}\label{eq:bound-alpha-E}
\alpha(n)=\frac{1}{\sqrt{2\pi}}\left(\frac{e}{d^{1/2}}\right)^d\,n^{(d-1)/2}.
\end{equation}
\end{lem}

\begin{proof}
As $Q_t$ is the uniform distribution on $\Ccal_t$, it assigns the probability $1/|\Ccal_t|$ to each sequence in $\Ccal_t$. One can easily compute $|\Ccal_t|$ as a multinomial coefficient
\begin{equation}\label{eq:bound-C-t}
|\Ccal_t|=\frac{n!}{t_1!\cdots t_d!}.
\end{equation}
We compare this to the probability $\pi_t^{\otimes n}(v)$ for $v\in\Ccal_t$, which is
\[\pi_t^{\otimes n}(v)= \frac{t_1^{t_1}\cdots t_d^{t_d}}{n^n}. \]
Hence we have, for any $v\in\Ccal_t$,
\[\frac{Q_t(v)}{\pi_t^{\otimes n}(v)}=\frac{n^n}{t_1^{t_1}\cdots t_d^{t_d}}\frac{t_1!\cdots t_d!}{n!}.\]
We now use Stirling's approximation formula~\cite{Robbins}, i.e.~for any positive integer $p$,
\begin{equation}\label{eq:Stirling-approx}
\sqrt{2\pi}\ p^{p+\frac12}e^{-p} \le p! \le e\ p^{p+\frac12}e^{-p}.
\end{equation}
This yields
\[ \frac{Q_t(v)}{\pi_t^{\otimes n}(v)} \le \frac{e^d}{\sqrt{2\pi}}\sqrt{\frac{t_1\cdots t_d}{n}} \le \frac{e^d}{\sqrt{2\pi}}\sqrt{\frac{t_1\cdots t_d}{n^d}}\,n^{(d-1)/2}. \]
To complete the proof we just have to show that $t_1\cdots t_d/n^d\le 1/d^d$. This follows from the arithmetic-geometric mean inequality, namely
\[ \left(\frac{t_1}{n}\times\cdots\times\frac{t_d}{n}\right)^{1/d} \le \frac{1}{d} \left(\frac{t_1}{n}+\cdots+\frac{t_d}{n}\right) = \frac{1}{d}, \]
where the last equality is because $t_1+\cdots+t_d=n$.
\end{proof}

\begin{remark}
The choice of the probability distributions $\pi_t$ in the result above is not arbitrary: using a standard Lagrange multiplier argument, one can show that $\sigma=\pi_t$ is the probability distribution on $V$ which maximizes $F(Q_t, \sigma^{\otimes n})$. Indeed, for a given type $t$, maximizing the fidelity is equivalent to maximizing the likelihood $\prod_{i=1}^d \sigma(i)^{t_i}$, under the constraint $\sum_{i=1}^d \sigma(i) = 1$. It is known that the maximum likelihood estimator for a multinomial distribution is given by the empirical frequencies, i.e.~$\sigma(i) = t_i/n$ for each $i=1,...,d$. 
\end{remark}

\begin{remark}
In the proof of the de Finetti Theorem for finite exchangeable probability distributions \cite[Lemma 6]{diaconis1980finite}, Diaconis and Freedman use a technical result to compare the hypergeometric and the multinomial distributions. Their bound has no relation to our bound though, since in the case where the number of variables on which the marginal is taken is equal to the initial number of variables, their bound becomes trivial.
\end{remark}

\begin{coro} \label{cor:exchangeable}
	For any exchangeable probability distribution $P \in \mathcal P(V^n)$, we have
	\[ P \le \mathrm{poly}(n) \sum_{t=1}^N \frac 1 N \, F(P, \pi_t^{\otimes n})^2 \, \pi_t^{\otimes n}, \]
	where $\mathrm{poly}(n)$ is a polynomial in $n$ of degree $2(d-1)$, $N=\binom{n+d-1}{n}$ and the probability distributions $\pi_t \in \mathcal P(V)$ have been defined in Lemma \ref{lem_exchangeability}.
\end{coro}

\begin{remark}
	Note that the degree of the polynomial pre-factor in Corollary \ref{cor:exchangeable} above (of order $d$) is better than the one obtained in \cite[Corollary 2.6]{lancien2016flexible} (of order $d^3$).
\end{remark}

We present next another way of deriving the result in Lemma \ref{lem_exchangeability}, giving a polynomial bound having the same degree. 

The main idea is the following map, modeled after the \emph{measure and prepare map} used in various proofs of the quantum de Finetti theorem \cite{chiribella2010quantum,harrow2013church}
$$\mathrm{MP}(Q) = \frac{(n+d-1)!}{n!} \int_{\pi \in \mathcal P(V)} F(Q, \pi^{\otimes n})^2 \pi^{\otimes n} d\pi,$$
where $d\pi$ is the Lebesgue measure on the simplex of probability vectors on $V$. Obviously, the range of the map $\mathrm MP$ is contained in the cone spanned by i.i.d.~random variables, $\mathbb R_+\operatorname{conv}\{\pi^{\otimes n}\}$.
Let us fix two arbitrary types $s,t$. First, let us show that the output $\mathrm{MP}(Q_t)$ is a well normalized probability distribution:
\begin{align*}
\frac{(n+d-1)!}{n!} \int_{\pi \in \mathcal P(V)} F(Q_t, \pi^{\otimes n})^2 d\pi & = \frac{(n+d-1)!}{n!} \binom{n}{t} \int_{\pi \in \mathcal P(V)} \prod_{i=1}^d \pi(i)^{t_i} d\pi\\
& = \frac{(n+d-1)!}{n!} \binom{n}{t}  \frac{\prod_{i=1}^d t_i!}{(n+d-1)!} \\
& = 1,
\end{align*}
where we have used the fact that $F(Q_t, R)^2 = \binom{n}{t} R(v_0)$ for any exchangeable probability distribution $R$ and any vector $v_0$ in the class $\mathcal C_t$, and the integration formula for Dirichlet measures
$$\int_{\pi \in \mathcal P(V)} \prod_{i=1}^d \pi(i)^{t_i} d\pi = \frac{\prod_{i=1}^d t_i!}{(d-1+\sum_{i=1}^d t_i)!}.$$

Using the same two facts, we can compute
$$F(Q_s, \mathrm{MP}(Q_t))^2 = \binom{2n+d-1}{n}^{-1} \prod_{i=1}^d \binom{s_i+t_i}{s_i} =: \lambda_{st}.$$
Note that the matrix $(\lambda_{st})$ is bi-stochastic and symmetric: the symmetry is obvious from the formula above, while the fact that row sums and column sums are 1 follows from a classical combinatorial identity (see \cite[Theorem 2.20]{loehr2011bijective}). Hence, we have the following (unique) convex decomposition
$$\mathrm{MP}(Q_t) = \sum_{s \text{ type}} \lambda_{st} Q_s.$$
Writing $Q_t \leq \lambda_{tt}^{-1} \mathrm{MP}(Q_t)$, we have $Q_t \in \beta(n) \mathrm{conv}\{\pi^{\otimes n}\}$, where $\beta(n) = \max_{s\textbf{ type}} \lambda_{ss}^{-1}$. It follows from the log-convexity of the central binomial coefficients (see \cite[Section 3]{liu2007log}) that the (possibly non-integer) type $s$ attaining the maximum is the ``flat'' one, $s_i = n/d$, yielding the bound 
$$\beta(n) \leq \binom{2n+d-1}{n}  \binom{2n/d}{n/d}^{-d}.$$
An analysis using the Stirling approximation formula \eqref{eq:Stirling-approx} and some trivial bound on the central binomial coefficient, such as $\binom{2x}{x} \geq 4^x/\sqrt{4x}$, shows that $\beta(n)$ is bounded by a polynomial in $n$ of degree $(d-1)/2$, a result identical to the one in Lemma \ref{lem_exchangeability}.

\subsection{The case of Markov exchangeability}\label{sect:Markov}

We consider in this section the case of \emph{Markov exchangeability}, where each equivalence class can be parametrized by:
\begin{enumerate}
\item The first point $v_1 \in V$ where each sequence in the class starts.
\item The matrix of transitions $t=(t_{ij})_{i,j\in V}$ which encodes the number of transitions from $i$ to $j$ for each sequence in the class. For each $i\in V$, we shall write in the sequel $t_{i}$ for the total number of transitions leaving $i$, that is $t_{i}=\sum_{j\in V}t_{ij}$.
\end{enumerate}
For example, the following two sequences in $\{1,2,3\}^8$ are in the same equivalence class:
$$( 1 , 1 , 3 , 2 , 3 , 1 , 2 , 2) \sim (1 , 3 , 2 , 2 , 3 , 1 , 1 , 2 ).$$

A rough upper bound on the number $N$ of equivalence classes, using the fact that $\sum_{i,j\in V}{t_{ij}}=n-1$, is given by
\[N\le d\times \binom{n+d^2-1}{n}\,.\]
Note that again, $N$ is polynomial in $n$, namely $N\leq d\times (n+1)^{d^2-1}$.

\paragraph{}Before we prove a result analogous to Lemma \ref{lem_exchangeability}, we need to introduce some combinatorial machinery which will allow us to further generalize the notions of exchangeability and Markov exchangeability. Our proof techniques are motivated by an observation of Zaman \cite{Zaman1984urn}, connecting Markov exchangeability with the problem of Eulerian cycles in directed graphs and the BEST theorem. We refer the reader to \cite{tutte1941unicursal,van1951circuits} for the historical references and to \cite[Section 5.6] {stanley1999enumerative} for a modern treatment.

Recall that a directed multigraph $G=(V,E)$ is a graph having possibly multiple directed distinct edges between vertices (in other words, $E$ is a multiset of ordered pairs of vertices). An \emph{Eulerian cycle} of $G$ is a cyclic walk on $G$ containing each edge of $E$ exactly once. Recall that the \emph{outdegree} (resp.~the \emph{indegree}) of a vertex $v \in V$ is the number of edges $e \in E$ with initial (resp.~final) vertex $v$. We have the following characterization of Eulerian graphs \cite[Theorem 5.6.1]{stanley1999enumerative}:

\begin{theo}
	A directed multigraph $G$ has an Eulerian cycle if and only if it is connected, and for all vertices $v \in V$, $\operatorname{outdeg}(v) = \operatorname{indeg}(v)$. 
\end{theo}

The BEST theorem \cite{tutte1941unicursal,van1951circuits} is a remarkable combinatorial result, giving the number of Eulerian cycles in terms of the number of certain oriented spanning trees of $G$. We recall its statement below \cite[Theorem 5.6.2]{stanley1999enumerative}:

\begin{theo}
	Consider an Eulerian directed multigraph $G$ with a marked edge $e_0 \in E$ and a marked vertex $w_0 \in V$. Let $T(G,w_0)$ denote the number of spanning trees of $G$ oriented towards the vertex $w_0$ (i.e.~all orientations in the tree are pointing towards $w_0$). Then, the number of Eulerian cycles of $G$ starting with the edge $e_0$ is given by
	\[T(G,w_0) \prod_{i \in V} (\operatorname{outdeg}(i)-1)!\,.\]
\end{theo}

\begin{remark}
	It is a non-trivial fact that the number of oriented spanning trees $T(G, w_0)$ does not depend on the choice of the marked vertex $w_0$. We shall thus omit $w_0$ and simply write $T(G)$.
\end{remark}

In order to apply this theorem to our setting of Markov exchangeable probability distributions, we first need a ``trajectorial version'' of it, that is we only need to keep track of the order in which the vertices of a Eulerian cycle are visited. 

\begin{defi}
	An \emph{Eulerian trajectory} of a directed multigraph $G$ with a marked edge $e_0 \in E$ and a marked vertex $w_0 \in V$ is a cyclic sequence of vertices $(v_1, \ldots, v_n)$ with the properties that $v_1=w_0$ and that the number of times an ordered pair of vertices $(i,j)$ appears in the sequence $x$ as neighbouring elements (the sequence of vertices is cyclic, so $(v_n,v_1)$ are considered neighbouring) is equal to the number of oriented edges $i \to j$ in $G$.
\end{defi}

It is clear that each Eulerian trajectory corresponds bijectively to the following number of Eulerian cycles, by permuting the order in which the edges appear in the cycle:
\[\prod_{i,j \in V} |\{e \in E \, : \, e \text{ is an edge from $i$ to $j$, different than the marked edge $e_0$}\}|!\,.\]

We are now in measure to relate the results above to the notion of Markov exchangeability, as it was done in \cite{Zaman1984urn}. Recall that an equivalence class $\mathcal{C}_\tau$ for the relation of Markov exchangeability is given by a pair $\tau=(v,t)$, where $v \in V$ is a fixed vertex (the initial position) and $t=(t_{ij})_{i,j\in V}$ is a matrix counting the number of transitions from $i$ to $j$. From this data, one builds a directed multigraph $G = (V, E_t)$, with $t_{ij}$ oriented edges between vertices $i$ and $j$. Exactly one of the following two alternatives holds for the non-empty class $\mathcal C_\tau$:
\begin{enumerate}
	\item There is a unique vertex $w \in V$ with the property that $\operatorname{indeg}(w) = \operatorname{outdeg}(w) + 1$.
	\item $G$ is Eulerian. 
\end{enumerate}
Let us consider the vertex $w$ given by the former alternative and set $w=v$ in case the latter alternative holds. The letter $w$ is the last letter appearing in every word in $\mathcal C_\tau$. Note also that whenever $v \neq w$, we have $\operatorname{outdeg}(v) = \operatorname{indeg}(v) + 1$. Let us consider the modified graph $G_0 = (V, E_t \sqcup e_0)$, where $e_0$ is an extra edge $w \to v$, and mark the edge $e_0$ in $G_0$. The directed multigraph $G_0$ is Eulerian, so one can apply the BEST theorem to count the number of Eulerian cycles it has, and then derive from this its number of Eulerian trajectories. Note that Eulerian trajectories in $G_0$ correspond exactly to trajectories (in the Markov sense) in $\mathcal{C}_\tau$, so we conclude that 
\begin{equation}\label{eq:card-C-tau-ME}
|\mathcal{C}_\tau| = \frac{T(G_0)\prod_{i \in V} (\operatorname{outdeg}(i) - 1 + \mathbf 1_{i=w})!}{\prod_{i,j \in V} t_{ij}!} = t_w\, T(G_0)\, \frac{\prod_{i \in V} (t_i - 1)!}{\prod_{i,j \in V} t_{ij}!},
\end{equation}
where we recall that $t_i = \sum_{j\in V} t_{ij}  = \operatorname{outdeg}(i)$.

Let us work out on an example the computation from equation \eqref{eq:card-C-tau-ME}. Let $x=(11323122) \in \{1,2,3\}^8$ and consider $\mathcal C$ its equivalence class. The graph associated to $\mathcal C$ is depicted in the left panel of Figure \ref{fig:example-v-G}. The matrix $t$ for the class $\mathcal C$ reads
$$t = \begin{bmatrix}
1 & 1 & 1 \\
0 & 1 & 1 \\
1 & 1 & 0 
\end{bmatrix}$$
with the starting (resp.~ending) vertex $v=1$ (resp.~$w=2$). The graph $G_0$ is obtained from $G$ by adding an extra oriented edge $2 \to 1$. In $G_0$, there are three oriented spanning trees flowing towards (say) the vertex $1$. Moreover, we have $t_1=3$ and $t_2=t_3=2$. Hence, equation \eqref{eq:card-C-tau-ME} tells us that the size of $\mathcal C$ is
$$|\mathcal C| = t_2\,T(G_0)\,\frac{\prod_{i=1}^3(t_i-1)!}{\prod_{i,j=1}^3t_{i,j}!} = 2 \times 3 \times \frac{(3-1)![(2-1)!]^2}{(1!)^7(0!)^2} = 12.$$
Indeed, here are the 12 elements of the equivalence class $\mathcal C$ of $v$:
$$\begin{array}{ccccccccccccccccc}
1 & 1 & 3 & 2 & 3 & 1 & 2 & 2 & \qquad\qquad\qquad \qquad & 1 & 3 & 1 & 1 & 2 & 3 & 2 & 2 \\
1 & 1 & 3 & 2 & 2 & 3 & 1 & 2 & \qquad\qquad\qquad \qquad & 1 & 3 & 1 & 1 & 2 & 2 & 3 & 2\\
1 & 1 & 3 & 1 & 2 & 3 & 2 & 2 & \qquad\qquad\qquad \qquad & 1 & 2 & 3 & 1 & 1 & 3 & 2 & 2\\
1 & 1 & 3 & 1 & 2 & 2 & 3 & 2 & \qquad\qquad\qquad \qquad & 1 & 3 & 2 & 3 & 1 & 1 & 2 & 2\\
1 & 1 & 2 & 3 & 1 & 3 & 2 & 2 & \qquad\qquad\qquad \qquad & 1 & 2 & 2 & 3 & 1 & 1 & 3 & 2\\
1 & 1 & 2 & 2 & 3 & 1 & 3 & 2 & \qquad\qquad\qquad \qquad & 1 & 3 & 2 & 2 & 3 & 1 & 1 & 2
\end{array}$$

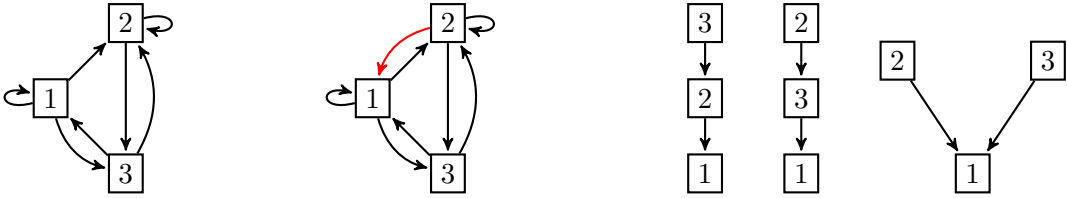
\begin{figure}[htpb]
	\begin{centering}
		\begin{tikzpicture}[->,>=stealth',shorten >=1pt,auto,node distance=3cm,
		thick,main node/.style={circle,draw,font=\sffamily\Large\bfseries}]
		
		\node[draw] at (0, 0) (1) {1};
		\node[draw] at (1,  1)  (2) {2};
		\node[draw] at (1, -1) (3) {3};
		
		\path[every node/.style={font=\sffamily\small}]
		(1) edge [loop left] node {} (1)
		edge [bend right] node[right] {} (3)
		edge [ right] node[right] {} (2)
		(2) edge node [right] {} (3)
		edge [loop right] node {} (2)
		(3) edge [bend right] node[right] {} (2)
		edge node [right] {} (1);
		\end{tikzpicture}\qquad\qquad
		\begin{tikzpicture}[->,>=stealth',shorten >=1pt,auto,node distance=3cm,
		thick,main node/.style={circle,draw,font=\sffamily\Large\bfseries}]
		
		\node[draw] at (0, 0) (1) {1};
		\node[draw] at (1,  1)  (2) {2};
		\node[draw] at (1, -1) (3) {3};
		
		\path[every node/.style={font=\sffamily\small}]
		(1) edge [loop left] node {} (1)
		edge [bend right] node[right] {} (3)
		edge [ right] node[right] {} (2)
		(2) edge [  bend right,color=red] node[right] {} (1)
		edge node [right] {} (3)
		edge [loop right] node {} (2)
		(3) edge [bend right] node[right] {} (2)
		edge node [right] {} (1);
		\end{tikzpicture}\qquad\qquad\qquad
		\begin{tikzpicture}[->,>=stealth',shorten >=1pt,auto,node distance=3cm,
		thick,main node/.style={circle,draw,font=\sffamily\Large\bfseries}]
		
		\node[draw] at (0, 0) (1) {1};
		\node[draw] at (0,  1)  (2) {2};
		\node[draw] at (0, 2) (3) {3};
		
		\path[every node/.style={font=\sffamily\small}]
		(2) edge node {} (1)
		(3) edge node {} (2);
		\end{tikzpicture}\qquad
		\begin{tikzpicture}[->,>=stealth',shorten >=1pt,auto,node distance=3cm,
		thick,main node/.style={circle,draw,font=\sffamily\Large\bfseries}]
		
		\node[draw] at (0, 0) (1) {1};
		\node[draw] at (0, 2)  (2) {2};
		\node[draw] at (0, 1) (3) {3};
		
		\path[every node/.style={font=\sffamily\small}]
		(2) edge node {} (3)
		(3) edge node {} (1);
		\end{tikzpicture}\qquad
		\begin{tikzpicture}[->,>=stealth',shorten >=1pt,auto,node distance=3cm,
		thick,main node/.style={circle,draw,font=\sffamily\Large\bfseries}]
		
		\node[draw] at (0, 0)  (1) {1};
		\node[draw] at (-1, 1.5)  (2)  {2};
		\node[draw] at (1, 1.5)  (3)  {3};
		
		\path[every node/.style={font=\sffamily\small}]
		(2) edge node {} (1)
		(3) edge node {} (1);
		\end{tikzpicture}
	\end{centering}
	\caption{Graphs $G$ and $G_0$ associated to $v=(11323122)$, as well as the three oriented trees flowing towards the vertex $1$.}
	\label{fig:example-v-G}
\end{figure}

We have now all the tools to prove the main technical result of this section, bounding the probability of an extremal Markov exchangeable probability by the probability of a Markov chain. 

\begin{lem}\label{lem:ME}
For all types $\tau=(v_1,t)$, we have
\begin{equation}\label{eq_lem_Markovexchangeability}
Q_\tau\le\alpha(n)\,\pi_\tau\,,
\end{equation}
where $\pi_\tau$ is the Markov distribution on $V$ starting in $v_1$ and transiting from $i$ to $j$ with probability $t_{i,j}/t_i$ for all $i,j=1,...,d$, and where $\alpha(n)$ is given by
\begin{equation}\label{eq:bound-alpha-ME}
\alpha(n)= \left(\frac{e^d}{\sqrt{2\pi}d^{d+1/2}}\right)^d (n-1)^{d(d+1)/2}\,.
\end{equation}
\end{lem}

\begin{proof}
	Using \eqref{eq:card-C-tau-ME}, we bound the size of the (non-empty) equivalence class $C_\tau$:
	\begin{equation}\label{eq:bound-C-tau}
	|\mathcal{C}_\tau|  \geq  \frac{\prod_{i \in V} (t_i - 1)!}{\prod_{i,j \in V} t_{ij}!}.
	\end{equation}
Now, we know from the proof of Lemma \ref{lem_exchangeability} that
\[  \frac{\prod_{j=1}^d t_{i,j}!}{t_i!} \le \frac{1}{\sqrt{2\pi}}\left(\frac{e}{d^{1/2}}\right)^{d}t_i^{(d-1)/2} \frac{\prod_{j=1}^d t_{i,j}^{t_{i,j}}}{t_i^{t_i}}\,. \]
Hence, we get, for some arbitrary $v \in \mathcal C_\tau$,
\begin{align*} Q_\tau(v)  = \frac{1}{|\mathcal C_\tau|}& \le \left(\frac{1}{\sqrt{2\pi}}\left(\frac{e}{d^{1/2}}\right)^{d}\right)^d \prod_{i=1}^d t_i^{(d+1)/2} \prod_{i=1}^d \frac{\prod_{j=1}^d t_{i,j}^{t_{i,j}}}{t_i^{t_i}} \\ 
& \le \left(\frac{e^d}{\sqrt{2\pi}d^{d+1/2}}\right)^d (n-1)^{d(d+1)/2} \prod_{i=1}^d \frac{\prod_{j=1}^d t_{i,j}^{t_{i,j}}}{t_i^{t_i}}\,, \end{align*}
where the last inequality holds because $\sum_{i=1}^d t_i=n-1$, so that by the arithmetic-geometric mean inequality, $\prod_{i=1}^d t_i\le \left((n-1)/d\right)^d$. Since for $v\in\mathcal{C}_\tau$ we also have 
\[ \pi_\tau(v)= \prod_{i=1}^d \frac{\prod_{j=1}^d t_{i,j}^{t_{i,j}}}{t_i^{t_i}}\,, \]
the proof is complete.
\end{proof}

\begin{coro} \label{cor:Markov-exchangeable}
	For any Markov exchangeable probability distribution $P \in \mathcal P(V^n)$, we have
	\[P \le \mathrm{poly}(n)  \sum_{\tau=1}^N \frac 1 N \, F(P, \pi_)^2 \, \pi_{\tau},\]
	where $\mathrm{poly}(n)$ is a polynomial in $n$ of degree $d(2d+1)-1$, $N\leq d\times \binom{n+d^2-1}{n}$ and the Markovian probability distributions $\pi_{\tau} \in \mathcal P(V^n)$ have been defined in Lemma \ref{lem:ME}.
\end{coro}

\subsection[The case of l-Markov exchangeability]{The case of $\ell$-Markov exchangeability}\label{sect:kMarkov}

In this subsection we study an extension of Markov exchangeability, where we consider distributions sampled from $(\ell+1)$-tensors, the Markovian case corresponding to the case $\ell=1$. A Markovian $(\ell+1)$-tensor on $V$ is a family $(p_{i_1,..,i_{\ell}}^{i_{\ell+1}})_{i_1,...,i_{\ell+1}=1}^d$ of real numbers such that:
\begin{enumerate}
\item $p_{i_1,..,i_{\ell}}^{i_{\ell+1}}\geq0$ for all $i_1,...,i_{\ell+1}$.
\item $\sum_{j=1}^d\,{p_{i_1,..,i_{\ell}}^{j}}=1$ for all $i_1,...,i_{\ell}$.
\end{enumerate}
For $n\geq \ell+1$, this defines a probability distribution on $V^n$ in the following way. Fix $x_1,...,x_{\ell}\in V$ as initial data. Then, for all $k\geq \ell+1$,
\[\Prob[x_{k+1}=j\,|\,x_1=i_1,...,x_k=i_k]=\Prob[x_{k+1}=j\,|\,x_{k-\ell+1}=i_{k-\ell+1},...,x_k=i_k]=p_{i_{k-\ell+1},..,i_k}^{j}\,.\]

When comparing to the usual notion of Markov exchangeability, we see that the ``future'' value of the process depends on the past only through the last $\ell$ values. 

We shall compare those distributions with the ones emerging from a form of partial exchangeability defined as follows. Given $x,y\in V^n$, we say that $x\sim y$ if $t_{i_1,..,i_{\ell}}^{i_{\ell+1}}(x)=t_{i_1,..,i_{\ell}}^{i_{\ell+1}}(y)$ for all $i_1,...,i_{\ell+1}$, where
\[t_{i_1,..,i_{\ell}}^{i_{\ell+1}}(x)=|\{\text{occurrences of the sequence }i_1,..,i_{\ell+1}\text{ in }x\}|\]
Remark that
\[\sum_{i_1,...,i_{\ell+1}=1}^d\,t_{i_1,..,i_{\ell}}^{t_{\ell+1}}=n-\ell\,.\]
We write $t_{i_1,...,i_\ell}=\sum_{i_{\ell+1}=1}^d\,t_{i_1,..,i_{\ell}}^{i_{\ell+1}}$. For this notion of partial exchangeability, the type is given by a $(\ell+1)$-tuple $\tau = (v_1,...,v_{\ell},t)$, where $v_1,...,v_{\ell}\in V$ are the initial data and $t=(t_{i_1,..,i_{\ell}}^{i_{\ell+1}})_{i_1,...,i_{\ell+1}=1}^d$ is the $(\ell+1)$-tensor whose elements are defined above. We can prove the following:

\begin{lem}\label{lem:kME}
For all types $\tau=(v_1,...,v_{\ell},t)$, we have
\begin{equation*}
Q_\tau\le\alpha(n)\,\pi_\tau,
\end{equation*}
where $\pi_\tau$ is the $\ell$-Markovian distribution on $V$ starting at $v_1,...,v_{\ell}$ and with transition $(\ell+1)$-tensor given by
\begin{equation}\label{eq:Markov-kernel-kME}
p_{i_1,..,i_{\ell}}^{i_{\ell+1}}=\frac{t_{i_1,...,i_{\ell}}^{i_{\ell+1}}}{t_{i_1,...,i_{\ell}}},\ i_1,...,i_{\ell+1}=1,\ldots,d,
\end{equation}
and where $\alpha(n)$ is given by
\begin{equation}\label{eq:bound-alpha-kME}
\alpha(n) = \left( \frac{e^d}{\sqrt{2\pi}d^{((\ell+1)d+\ell)/2}}\right)^{d^\ell} (n-\ell)^{d^\ell(d+1)/2}.
\end{equation}
\end{lem}

\begin{proof}
The proof is an extension of the one for Markov exchangeability, which corresponds to the case $\ell=1$. In particular, the result will follow from an application of the BEST Theorem. We want to lower bound the cardinal of the equivalence classes $\Ccal_\tau$ for a fixed type $\tau$. To this end, we explicit a one-to-one correspondence between this set and the Eulerian trajectories of some directed multigraph. Consider the graph $G=(V_1\times\cdots\times V_{\ell},E)$ defined by the following properties:
\begin{enumerate}
\item An oriented edge can possibly exist from $(i_1,...,i_{\ell})$ to $(j_1,...,j_{\ell})$ only if $(i_2,...,i_{\ell})=(j_1,...,j_{\ell-1})$.
\item The number of edges $(i_1, \ldots, i_{\ell}) \to (i_2, \ldots, i_{\ell+1})$ is given by $t_{i_1,..,i_{\ell}}^{i_{\ell+1}}$.
\end{enumerate}
As in the proof of Lemma \ref{lem:ME}, we see that the number of Eulerian trajectories starting at the vertex $(v_1,...,v_\ell)$ for this graph is exactly the cardinal of $\Ccal_\tau$. Remark also that the probability distribution $\pi_\tau$ is the one obtained from the simple walk on $G$ starting at $(v_1,...,v_d)$. Then, following the same argument as in the Markov case, we obtain that
$$|\mathcal{C}_\tau|  \geq  \frac{\prod_{i \in V} (t_{i_1, \ldots, i_{\ell}} - 1)!}{\prod_{i_1, \ldots, i_{\ell+1} \in V} t_{i_1, \ldots, i_{\ell}}^{i_{\ell+1}}!}.$$
The end of the proof is exactly the same as the one of the Markov case: we bound the ratio between the lower bound in the previous equation and the probability coming from the Markov chain formula using the $(\ell+1)$-tensor in \eqref{eq:Markov-kernel-kME}. We leave the details to the reader.
\end{proof}

\begin{remark}
The bound from \eqref{eq:bound-alpha-kME} obviously reduces to \eqref{eq:bound-alpha-ME} in the case of the usual Markov exchangeability, when $\ell=1$. Notice also that \eqref{eq:bound-alpha-kME} resembles the bound for exchangeability from \eqref{eq:bound-alpha-E}, when $\ell=0$. Actually, the only difference between the two expressions for $\alpha(n)$ is a factor of $n$, which comes from the fact that in the general case of $\ell$-Markov exchangeability we can only lower bound the cardinality of the equivalence classes $\mathcal C_\tau$ (see equation \eqref{eq:bound-C-tau}). In the exchangeable situation, the cardinality of an equivalence class is exactly a binomial coefficient (see equation \eqref{eq:bound-C-t}). We lose a factor of $n$ between the lower bound \eqref{eq:bound-C-tau} and the exact expression \eqref{eq:bound-C-t}.
\end{remark}

\begin{coro} \label{cor:lMarkov-exchangeable}
	For any $\ell$-Markov exchangeable probability distribution $P \in \mathcal P(V^n)$, we have
	\[P \le \mathrm{poly}(n)  \sum_{\tau=1}^N \frac 1 N \, F(P, \pi_{\tau})^2 \, \pi_{\tau},\]
	where $\mathrm{poly}(n)$ is a polynomial in $n$ of degree $d^{\ell}(2d+1)-1$, $N\leq d^{\ell}\times \binom{n+d^{\ell+1}-\ell}{n-\ell+1}$ and the $\ell$-Markovian probability distributions $\pi_{\tau} \in \mathcal P(V^n)$ have been defined in Lemma \ref{lem:kME}.
\end{coro}

\subsection{The case of double exchangeability}\label{sect:double}

The general setting of partially exchangeable distributions also allows to study finer forms of exchangeability when more is known on the structure of the space $V$. We treat in this section the case where $V=V_1\times V_2$ for two finite sets $V_1$ and $V_2$, but the results can be easily generalized to more than two parties. Also, we focus here on the notions of double exchangeability and double Markov exchangeability, but one can come up with many potential generalizations: for instance double $\ell$-Markov exchangeability, or exchangeability on $V_1$ and Markov exchangeability on $V_2$ etc. We leave these extensions as exercises.

Recall the definition of the Cartesian product of two equivalence relations: if $R_i$ is an equivalence relation on $V_i$, $i=1,2$, the Cartesian product $R := R_1 \times R_2$ is an equivalence relation on $V_1 \times V_2$ given by
$$(x, y)R(x',y') \iff x R_1 x' \text{ and } y R_2 y'.$$

We have the following result, showing that de Finetti reductions are stable by Cartesian products. 

\begin{prop}
For $i=1,2$, consider a sequence $R_i(n)$ of equivalence realtions on $\mathcal P(V_i^n)$ together with a sequence of distinguished subsets $\Pi_i(n) \subseteq \mathcal P_{R_i(n)}(V_i^n)$ of $R_i(n)$--exchangeable probability distributions, such that the pair $(R_i,\Pi_i)$ admits a de Finetti reduction with polynomial $\beta_i(n)$. Then the pair $(R_1 \times R_2, \Pi_1 \otimes \Pi_2)$ admits a de Finetti reduction with polynomial pre-factor $\beta(n) = \beta_1(n) \beta_2(n)$. 
\end{prop}

\begin{proof}
First, we claim that the extremal points of $R_1 \times R_2$--exchangeable distributions are precisely the uniform distributions on $(\mathcal C_i \times \mathcal D_j)_{(i,j) \in I \times J}$, where $(\mathcal C_i)_{i \in I}$, resp.~$(\mathcal D_j)_{j \in J}$, are the equivalence classes on $V_1^{\times n}$ induced by $R_1(n)$, resp.~the equivalence classes on $V_2^{\times n}$ induced by $R_2(n)$. 	This claim follows directly from the definition of the Cartesian product of equivalence relations: indeed, $(\mathcal C_i \times \mathcal D_j)_{(i,j) \in I \times J}$ are exactly the equivalence classes of $R_1(n) \times R_2(n)$. The result follows now from multiplying the de Finetti reduction estimates \eqref{eq:def-deF-reduction} for the uniform probability distributions on the classes $\mathcal C_i$ and $\mathcal D_j$.
\end{proof}

Before looking at the incarnations of the above general statement in the cases of exchangeability and Markov exchangeability, let us consider a simple example. Let $V_1 = V_2 = \{1,2\}$ and $n=3$. Consider the equivalence relations $R_1 = R_2$ corresponding to the usual exchangeability on $V_1^3$ and $V_2^3$. Then, the sequences 
$$v^1 = ((1,1), (2,1), (2,2)) \quad \text{ and } \quad v^2 = ((1,2), (2,1), (2,1))$$
are in the same equivalence class for the Cartesian product equivalence relation. Indeed, they have the same type $t = (t^1,t^2)$, where $t^1_1=1$, $t^1_2 = 2$ and $t^2_1=2$, $t^2_2=1$. The sequences $v^1$ and $v^2$ are however not equivalent with respect to the \emph{global exchangeable} equivalence relation on $\tilde V = \{11,12,21,22\}$. It is important in the following corollaries to make the distinction between these two equivalence relations on $V = V_1 \times V_2$.

\begin{coro}\label{theo:doubly_exch}
For all types $t = \left(t^{1},t^{2}\right)$, we have
\begin{equation}\label{eq1_lem_doubleexchangeability}
Q_t\le\alpha(n)\,\pi_{t^{1}}^{\otimes n}\otimes\pi_{t^{2}}^{\otimes n},
\end{equation}
where $\pi_{t^{r}}$ is the probability distribution on $V_{r}$ such that $\pi_{t^{r}}(i)=t^{r}_i/n$ for all $i\in V_{r}$, $r=1,2$, and where $\alpha(n)$ is given by
\begin{equation}\label{eq:bound-alpha-doublyE}
\alpha(n)=\frac{1}{2\pi}\left(\frac{e}{d_1^{1/2}}\right)^{d_1} \left(\frac{e}{d_2^{1/2}}\right)^{d_2} n^{(d_1+d_2-2)/2}.
\end{equation}
\end{coro}

\begin{coro} \label{cor:doubly-exchangeable}
	For any doubly exchangeable probability distribution $P \in \mathcal P(V^n)$, we have
	$$P \le \mathrm{poly}(n) \sum_{\substack{1\leq t^1\leq N_1 \\ 1\leq t^2\leq N_2}} \frac{1}{N_1N_2} \, F(P, \pi_{t^1}^{\otimes n}\otimes\pi_{t^2}^{\otimes n})^2 \, \pi_{t^1}^{\otimes n}\otimes\pi_{t^2}^{\otimes n},$$
	where $\mathrm{poly}(n)$ is a polynomial in $n$ of degree $2(d_1+d_2-2)$, $N_r=\binom{n+d_r-1}{n}$, $r=1,2$, and the probability distributions $\pi_{t^r} \in \mathcal P(V_r)$, $r=1,2$, have been defined in Corollary \ref{theo:doubly_exch}.
\end{coro}

We can similarly define the notion of double Markov exchangeability. This time the types are given by $\tau=(\tau^{1},\tau^{2})$, where $\tau^{i}$ is a type for Markov exchangeability on $V_i$, $i=1,2$.

\begin{coro}\label{theo:doubly_Markov_exch}
For all types $\tau=\left(\tau^{1}=(v^{1}_1,t^{1}),\tau^{2}=(v^{2}_1,t^{2})\right)$, we have
\begin{equation}\label{eq_lem_doublyMarkovexchangeability}
Q_\tau\le\alpha(n)\,\pi_{\tau^{1}}\otimes\pi_{\tau^{2}},
\end{equation}
where $\pi_{\tau^{r}}$ is the Markov distribution on $V_r$ starting in $v^{r}_1$ and transiting from $i$ to $j$ with probability $t_{i,j}^{r}/t_i^{r}$ for all $i,j\in V_r$, $r=1,2$, and where $\alpha(n)$ is given by
\begin{equation}\label{eq:bound-alpha-doublyME}
\alpha(n)= \left(\frac{1}{2\pi}\right)^{(d_1+d_2)/2} \left(\frac{e^{d_1}}{d_1^{d_1+1/2}}\right)^{d_1} \left(\frac{e^{d_2}}{d_2^{d_2+1/2}}\right)^{d_2}  (n-1)^{(d_1(d_1+1)+d_2(d_2+1))/2}.
\end{equation}
\end{coro}

\begin{coro} \label{cor:doubly-Markov-exchangeable}
	For any doubly Markov exchangeable probability distribution $P \in \mathcal P(V^n)$, we have
	$$P \le \mathrm{poly}(n) \sum_{\substack{1\leq \tau^1\leq N_1 \\ 1\leq \tau^2\leq N_2}} \frac{1}{N_1N_2} F(P, \pi_{\tau^1}\otimes\pi_{\tau^2})^2 \pi_{\tau^1}\otimes\pi_{\tau^2},$$
	where $\mathrm{poly}(n)$ is a polynomial in $n$ of degree $d_1(2d_1+1)+d_2(2d_2+1)-2$, $N_r=d_r\times\binom{n+d_r^2-1}{n}$, $r=1,2$, and the Markovian probability distributions $\pi_{\tau^r} \in \mathcal P(V_r^n)$ have been defined in Corollary \ref{theo:doubly_Markov_exch}.
\end{coro}

\section{Universal de Finetti reductions for exchangeable conditional probability distributions} \label{sect:conditional}

Our aim now is to show that exchangeability also obeys a universal conditional de Finetti reduction, in the sense of Definition \ref{def:deF-reduction-cond}. As before, we first state a general result, which we then apply to our case of focus.  

\subsection{General statement}

We start with proving an equivalent of Lemma \ref{lem_constrainedDeFinettiTheorem} for partially exchangeable conditional probability distributions, which is as the latter valid for any notion of partial exchangeability. Before going into the precise result, it might be worth insisting on one point. Lemma \ref{lem_conditionedDeFinetti} below is a statement about a conditional probability distribution $P_{A^n|X^n}$ obtained from a $\sim$--exchangeable probability distribution $P_{A^nX^n}$ on $A^n\times X^n$. It might happen though that the object at hand is only a $\sim$--exchangeable conditional probability distribution $P_{A^n|X^n}$ (what we mean by this is that $P(a|x)=P(b|y)$ whenever $(a,x)\sim(b,y)$). Nevertheless, it is easy in this case to construct a $\sim$--exchangeable probability distribution $P'_{A^nX^n}$ which is such that $P'_{A^n|X^n}=P_{A^n|X^n}$: one simply has to set $P'_{A^nX^n}=Q_{X^n}P_{A^n|X^n}$ with $Q_{X^n}$ satisfying $Q(x)=Q(y)$ whenever there exist $a,b$ such that $(a,x)\sim(b,y)$. Lemma \ref{lem_conditionedDeFinetti} thus applies to $P'_{A^nX^n}$, and therefore also directly to $P_{A^n|X^n}$, without any knowledge of a $\sim$--exchangeable probability distribution of origin. 

\begin{lem}\label{lem_conditionedDeFinetti}
	Suppose that there exist functionals $\alpha_1(n),...,\alpha_N(n)$, $\alpha'_1(n),...,\alpha'_N(n)$ and probability distributions $\pi_k\in \Pcal_\sim(A^n\times X^n)$, $k=1,...,N$ such that
	\begin{equation}\label{eq1_lem_conditionedDeFinetti}
	Q_{k,A^nX^n}\le\alpha_k(n)\,\pi_{k,A^nX^n}\ \ \text{and}\ \ \pi_{k,X^n}\le\alpha'_k(n)\,Q_{k,X^n}\ \text{on}\ \mathrm{supp}(Q_{k,X^n}) \qquad \text{ for all } k=1,...,N.
	\end{equation}
	Then, for any probability distribution $P\in\Pcal_\sim(A^n\times X^n)$, one has
	\begin{equation}\label{eq2_lem_conditionedDeFinetti}
	P_{A^n|X^n}\le N\times\alpha(n)\times\alpha'(n)\sum_{k=1}^N\, \frac{1}{N}\,\pi_{k,A^n|X^n}
	\end{equation}
	where $\alpha(n)=\max\,\{\alpha_k(n),\ k=1,...,N\}$, $\alpha'(n)=\max\,\{\alpha'_k(n),\ k=1,...,N\}$.
	
	In particular, if $\alpha(n),\alpha'(n)$ are polynomial in $n$, then the pair $(\sim, \{\pi_{1,A^n|X^n},\ldots,\pi_{N,A^n|X^n}\})$, admits a universal de Finetti reduction in the sense of Definition \ref{def:deF-reduction-cond}, with $\mathrm{poly}(n)=N\times\alpha(n)\times\alpha'(n)$.
\end{lem}

\begin{proof}
	Since $P_{A^nX^n}\in\Pcal_\sim(A^n\times X^n)$, we have as in the proof of Lemma \ref{lem_constrainedDeFinettiTheorem}, the equality
	\[ P_{A^nX^n} =\sum_{k=1}^N{\mu_k(P_{A^nX^n})Q_{k,A^nX^n}},\]
	where $\mu_k(P_{A^nX^n})=F(P_{A^nX^n},Q_{k,A^nX^n})^2$. Hence, for each $k=1,\ldots,N$, we obviously have
	\[\mu_k(P_{A^nX^n})Q_{k,A^nX^n}\le P_{A^nX^n},\]
	so that, by taking the marginal over $A^n$,
	\[\mu_k(P_{A^nX^n})Q_{k,X^n}\le P_{X^n}.\]
	And we therefore get 
	\[P_{A^n|X^n}\le \sum_{k=1}^N\, Q_{k,A^n|X^n}.\]
	Now we exploit inequality \eqref{eq1_lem_conditionedDeFinetti} to argue that, for each $k=1,...,N$, 
	\[Q_{k,A^n|X^n}\le \alpha_k(n)\times\alpha'_k(n)\,\pi_{k,A^n|X^n},\]
	which concludes the proof. 
\end{proof}

\begin{remark}
	As already mentioned in the introduction, the main difference between the results of Lemmas \ref{lem_constrainedDeFinettiTheorem} and \ref{lem_conditionedDeFinetti} is that the right hand side of equation \eqref{eq2_lem_constrainedDeFinettiTheorem} depends on $P$ (via the fidelity weight) while the right hand side of equation \eqref{eq2_lem_conditionedDeFinetti} is independent of $P_{A^n|X^n}$. This is an issue which was already pointed out in \cite{L-W}, in the particular case of exchangeability: in de Finetti reductions for conditional probability distributions, any information that one may have on the $P_{A^n|X^n}$ of interest (other than its partial exchangeability) is not reflected in the upper bounding conditional probability distribution.
\end{remark}

\subsection{Specific result}

Before being able to state the counterpart of Corollary \ref{cor:exchangeable} for exchangeable conditional probability distributions, we need one last technical lemma.  

\begin{lem} \label{lem:exch-marginal}
	The marginals on $X^n$ of exchangeable probability distributions on $A^n\times X^n$ are exchangeable probability distributions on $X^n$. What is more, the marginals on $X^n$ of the extreme exchangeable probability distributions on $A^n\times X^n$ are exactly the extreme exchangeable probability distribution on $X^n$, with some redundancies.
\end{lem}

\begin{proof}
Let us start with one crucial observation. Given $x,y\in X^n$, the following statements are equivalent: 
\begin{enumerate}
\item $x$ and $y$ are exchangeable.
\item For any $a\in A^n$, there exists $b\in B^n$ such that $(a\,x)$ and $(b\,y)$ are exchangeable. 
\item There exists $\phi:A^n\rightarrow A^n$ bijective such that, for any $a\in A^n$, $(a\,x)$ and $(\phi(a)\,y)$ are exchangeable. 
\end{enumerate}
With this in mind, the two assertions in Lemma \ref{lem:exch-marginal} are now easy to prove. 

Let $P_{A^nX^n}$ be an exchangeable probability distribution on $A^n\times X^n$ and $P_{X^n}$ be its marginal on $X^n$. Assume that $x,y\in X^n$ are exchangeable. Then,
\[ P_{X^n}(x) = \sum_{a\in A^n} P_{A^nX^n}(a\,x) = \sum_{a\in A^n} P_{A^nX^n}(\phi(a)\,y) = P_{X^n}(y), \]
where the first and last equalities are by definition of the marginal $P_{X^n}$ while the middle one is by exchangeability of $P_{A^nX^n}$. This proves that $P_{X^n}$ is indeed an exchangeable probability distribution on $X^n$.

Let $Q_{A^nX^n}$ be an extreme exchangeable probability distribution on $A^n\times X^n$ and $Q_{X^n}$ be its marginal on $X^n$. Let $x,y\in X^n$ and assume that $Q_{X^n}(x)\neq 0$. Suppose first that $x$ and $y$ are exchangeable. Since we know by what precedes that $Q_{X^n}$ is exchangeable, we then have $Q_{X^n}(y)=Q_{X^n}(x)$, and thus $Q_{X^n}(y)\neq 0$. Suppose now that $x$ and $y$ are not exchangeable. Since $Q_{X^n}(x)=\sum_{a\in A^n} Q_{A^nX^n}(a\,x)$, we know that there exists $a\in A^n$ such that $Q_{A^nX^n}(a\,x)\neq 0$. So if we had $Q_{X^n}(y)\neq 0$, then by the same reasoning there would exist $b\in A^n$ such that $Q_{A^nX^n}(b\,y)\neq 0$. Because $Q_{A^nX^n}$ is an extreme exchangeable probability distribution, this would imply that $(a\,x)$ and $(b\,y)$ are exchangeable, and therefore that $x$ and $y$ are exchangeable, which is a contradiction. Hence necessarily, $Q_{X^n}(y)= 0$. This proves that $Q_{X^n}$ is indeed an extreme exchangeable probability distribution on $X^n$.
\end{proof}

\begin{remark} 
	Note that Lemma \ref{lem:exch-marginal} is really specific to exchangeability, and is not necessarily true for any notion of partial exchangeability. For instance, it is false for Markov exchangeability. A counter-example can be provided already in the simplest case where $A=X=\{1,2\}$. Consider the sequence $(ax)=((11),(12),(21),(21))\in A^4\times X^4$. It is the only element in its equivalence class for Markov exchangeability on $A^4\times X^4$, characterized by $(11)$ as first point and transitions $(11)\rightarrow(12)$, $(12)\rightarrow(21)$, $(21)\rightarrow(21)$ appearing once each. The uniform distribution $Q_{A^4X^4}$ on this class is thus the Dirac mass on $(ax)$, whose marginal $Q_{X^4}$ is the Dirac mass on $x=(1,2,1,1)$. Now, the equivalence class of $x$ for Markov exchangeability on $X^4$, which is characterized by $1$ as first point and the transitions $1\rightarrow 2$, $2\rightarrow 1$, $1\rightarrow 1$ appearing once each, also contains $y=(1,1,2,1)$. So the uniform distribution on this class is not $Q_{X^4}$ (and in fact $Q_{X^4}$ is not even Markov exchangeable).
\end{remark}

In the case of exchangeability, the values of $N$, the number of classes, and of $\alpha(n)$, the maximal ratio $Q_{k,A^nX^n}/\pi_{k,A^nX^n}$, appearing in Lemma \ref{lem_conditionedDeFinetti}, have already been shown to be polynomial in $n$ (in Subsection \ref{sect:exch}). So the only thing that remains to be done is to show that also $\alpha'(n)$, the maximal ratio $\pi_{k,X^n}/Q_{k,X^n}$ on the support of $Q_{k,X^n}$, is polynomial in $n$. Now, by Lemma \ref{lem:exch-marginal}, we know that the marginals $\pi_{k,X^n}$ are exchangeable distributions and the marginals $Q_{k,X^n}$ are extreme exchangeable distributions. So the value of the maximal ratio $\pi_{k,X^n}/Q_{k,X^n}$ on the support of $Q_{k,X^n}$ is also provided by the proof of Lemma \ref{lem_exchangeability}, and is actually $1$.
With these observations at hand, we can then immediately derive from Lemma \ref{lem_conditionedDeFinetti} Corollary \ref{cor:exchangeable-cond} below, in which we denote by $d$ the size of the alphabet $A\times X$.

\begin{coro} \label{cor:exchangeable-cond}
	For any exchangeable probability distribution $P \in \mathcal P(A^n\times X^n)$, we have
	\[ P_{A^n|X^n} \le \mathrm{poly}(n) \sum_{t=1}^N \frac 1 N \,\pi^{\otimes n}_{t,A^n|X^n}, \]
	where $\mathrm{poly}(n)$ is a polynomial in $n$ of degree $3(d-1)/2$, $N=\binom{n+d-1}{n}$ and the probability distributions $\pi_t \in \mathcal P(A\times X)$ have been defined in Lemma \ref{lem_exchangeability}.
\end{coro}

\begin{remark}
	Note that the degree of the polynomial pre-factor in Corollary \ref{cor:exchangeable} above (roughly $d$) is the same as the one obtained in \cite[Corollary 3]{A-FR15}.
\end{remark}

Note that in the cases of Markov exchangeability and $\ell$-Markov exchangeability, it has already been shown as well that the values of $N$, the number of classes, and of $\alpha(n)$, the maximal ratio $Q_{k,A^nX^n}/\pi_{k,A^nX^n}$, appearing in Lemma \ref{lem_conditionedDeFinetti}, are polynomial in $n$ (in Subsections \ref{sect:Markov} and \ref{sect:kMarkov}). In these cases though, since Lemma \ref{lem:exch-marginal} does not hold, the proofs of Lemmas \ref{lem:ME} and \ref{lem:kME} do not tell us that $\alpha'(n)$, the maximal ratio $\pi_{k,X^n}/Q_{k,X^n}$ on the support of $Q_{k,X^n}$, is $1$. It could however be that it is indeed also polynomial in $n$, but we were not able to prove this by another approach.

\section{Application: winning probability in repeated two-player non-local games}\label{sect:NLgames}

Let us describe the general setting that we have in mind for applying the results above, which is the one of two-player non-local games. Such a game $G$ is played by two separated but cooperating players. Each of them receives an input $x\in X$ or $y\in Y$, and has to produce an output $a\in A$ or $b\in B$. The players are declared to have won the game if a given binary predicate $V(xyab)$ on their two inputs and two outputs is satisfied. To achieve this goal they are allowed to agree on a common strategy before the game starts, depending on the probability distribution of the inputs $T_{XY}$ and on the binary predicate $V$, which are both known by them. But once the game has started they cannot communicate any more. The strategies (i.e.~conditional probability distributions) $P_{AB|XY}$ that they are able to implement are therefore restricted by this locality constraint. What is precisely considered their allowed set of strategies depends on how much physical power they are assumed to have: they could for instance be sharing classical randomness, or quantum entanglement, or non-signalling boxes, or some other type of correlations. The quantity that we are then interested in their maximum winning probability, i.e.
\[ \omega(G)= \max\left\{ \sum_{x\in X,y\in Y}T(xy)\sum_{a\in A,b\in B} V(xyab)P(ab|xy),\ P\ \text{allowed} \right\}. \]

In the sequel, we will actually look at what happens when such two-player non-local game $G$ is repeated $n$ times, either in parallel or in sequence. The question we will ask is the following: if the players cannot win $1$ instance of the game $G$ with probability $1$, does there probability of winning $n$ instances of $G$ (played either simultaneously or one after another) decay exponentially to $0$ as $n$ grows? We will describe a potential strategy to tackle this problem, which is exactly the one pushed through in \cite{L-W} in the case of the parallel repetition of $G$ for non-signalling players.

\subsection{Parallel repetition of two-player non-local games}\label{sect:parallel}

The parallel repetition of $G$, that we denote $G^n_{par}$, consists in each player receiving simultaneously $n$ inputs $(x_1,\ldots,x_n)\in X^n$ or $(y_1,\ldots,y_n)\in Y^n$, and having to produce simultaneously $n$ outputs $(a_1,\ldots,a_n)\in A^n$ or $(b_1,\ldots,b_n)\in B^n$. The players are then declared to have won if they have won all $n$ instances of $G$, i.e.~if the $n$ binary predicates $V(x_1y_1a_1b_1),\ldots,V(x_ny_na_nb_n)$ are satisfied. Concerning the probability distribution of the inputs, we simply assume that the $n$ pairs of inputs are sampled independently according to the probability $T_{XY}$. The maximum winning probability when playing $G^n_{par}$ is thus
\[ \omega(G^n_{par}) = \max \left\{ \sum_{x\in X^n,y\in Y^n}T^{\otimes n}(xy) \sum_{a\in A^n,b\in B^n} V^{\otimes n}(xyab)P^{n}(ab|xy),\ P^{n}\ \text{allowed} \right\}. \]
What we would like now is to compare $\omega(G^n_{par})$ to $(\omega(G))^n$. Indeed, if the optimal strategy $P^n_*$ for $G^n_{par}$ were of the form $P^{\otimes n}_*$ with $P_*$ an optimal strategy for $G$, then we would clearly have $\omega(G^n_{par})=(\omega(G))^n$. But it could be that some allowed strategies $P^n$ which are not of the form $P^{\otimes n}$ with $P$ an allowed strategy give to the players a higher winning probability in $G^n_{par}$.

Obviously, the function $(x,y,a,b)\in X^n\times Y^n\times A^n\times B^n\mapsto T^{\otimes n}(xy)V^{\otimes n}(xyab)$ is exchangeable in the $n$ copies of $X\times Y\times A\times B$. So we can assume without loss of generality that the optimal strategy $P^{n}_*$ for $G^n_{par}$ is also exchangeable. And hence by Corollary \ref{cor:exchangeable}, we can upper bound $\omega(G^n_{par})$ as
\[ \omega(G^n_{par}) \leq \mathrm{poly}(n) \sum_{t=1}^N \frac{1}{N} F(T^{\otimes n}P^{n}_*,\pi_t^{\otimes n})^2 {\langle V,\pi_t \rangle}^n. \]

Note that in the expression above, the probabilities $\pi_t$'s are a priori not of the form $TP$ with $P$ an allowed strategy, so that it might be that $\langle V,\pi_t \rangle > \omega(G)$. However, we can split the $\pi_t$'s into two groups: those that are $\delta$-close to being of the form $TP$ with $P$ an allowed strategy, for which we have $\langle V,\pi_t \rangle^n\leq(\omega(G)+\delta)^n$, and the other ones, for which we hope to be able to show that $F(T^{\otimes n}P^{n}_*,\pi_t^{\otimes n})^2\leq (1-f(\delta))^n$.

\subsection{Sequential repetition of two-player non-local games}\label{sect:sequential}

The sequential repetition of $G$, that we denote $G^n_{seq}$, consists in each player receiving one after another $n$ inputs $x_1,\ldots,x_n\in X$ or $y_1,\ldots,y_n\in Y$, and having to produce one after another $n$ outputs $a_1,\ldots,a_n\in A$ or $b_1,\ldots,b_n\in B$. As in the case of the parallel repetition of $G$, the players are then declared to have won if they have won all $n$ instances of $G$. But concerning the probability distribution of the inputs, we this time assume the following: the first pair of inputs is sampled according to the probability $T_{XY}$, while subsequent pairs of inputs are sampled from the one preceding them according to a transition probability $\tau_{XYXY}$ such that, for each $x\in X,y\in Y$, $\sum_{x'\in X,y'\in Y}\tau(x'y',xy)=T(xy)$. The maximum winning probability when playing $G^n_{seq}$ is thus
\[ \omega(G^n_{seq}) = \max \left\{ \sum_{x\in X^n,y\in Y^n} T\otimes\tau^{\otimes (n-1)}(xy)\sum_{a\in A^n,b\in B^n} V^{\otimes n}(xyab)P^{n}(ab|xy),\ P^{n}\ \text{allowed} \right\}. \]
What we would like now is to compare $\omega(G^n_{seq})$ to $(\omega(G))^n$. Indeed, if the optimal strategy $P^n_*$ for $G^n_{seq}$ were of the form $P^{\otimes n}_*$ with $P_*$ an optimal strategy for $G$, then we would clearly have $\omega(G^n_{seq})\leq(\omega(G))^n$. But, as for $G^n_{par}$, it could be that some allowed strategies $P^n$ which are not of the form $P^{\otimes n}$ with $P$ an allowed strategy give to the players a higher winning probability in $G^n_{seq}$.

Obviously, the function $(x,y,a,b)\in X^n\times Y^n\times A^n\times B^n\mapsto T\otimes\tau^{\otimes (n-1)}(xy) V^{\otimes n}(xyab)$ is Markov exchangeable in the $n$ copies of $X\times Y\times A\times B$. So we can assume without loss of generality that the optimal strategy $P^{n}_*$ for $G^n_{seq}$ is also Markov exchangeable. And hence by Corollary \ref{cor:Markov-exchangeable}, we can upper bound $\omega(G^n_{seq})$ as
\[ \omega(G^n_{seq}) \leq \mathrm{poly}(n) \sum_{k=1}^N \frac{1}{N} F(T\otimes\tau^{\otimes (n-1)}P^{n}_*,\pi_k)^2 {\langle V^{\otimes n},\pi_k \rangle}. \]

We will once again argue similarly to what we did in the parallel repetition case. In the expression above, the Markov chains $\pi_k$'s are a priori not of the form $T\otimes\tau^{\otimes (n-1)}P^{\otimes n}$ with $P$ an allowed strategy, so that it might be that ${\langle V^{\otimes n},\pi_k \rangle} > (\omega(G))^n$. However, we can split the Markov chains $\pi_k$'s into two groups: those whose transition probability is $\delta$-close to being of the form $\tau P$ with $P$ an allowed strategy, for which we have ${\langle V^{\otimes n},\pi_k \rangle}\leq(\omega(G)+\delta)^n$, and the other ones, for which we hope to be able to show that $F(T\otimes\tau^{\otimes (n-1)}P^{n}_*,\pi_k)^2\leq (1-f(\delta))^n$.

\bigskip

\noindent \textit{Acknowledgments.} I.B.~C.L.~and I.N.~are supported by the French CNRS (ANR project StoQ ANR-14-CE25-0003). C.L.~is supported by the European Research Council (grant 64891) and the Spanish MINECO (grant MTM2014-54240-P). I.N.~acknowledges support from  the von Humboldt foundation, the PHC Sakura program (project 38615VA) and the French CNRS (InFIniTi project MISTEQ). The authors would like to thank the Institut Henri Poincar\'e and its staff for its hospitality during the trimester T3-2017 ``Analysis in Quantum Information Theory'', where this project was finalized. 

\bibliographystyle{alpha}
\bibliography{biblio}
\end{document}

%% file: packages.tex
\usepackage{amsmath} 
\usepackage{amssymb} 
\usepackage{fullpage}
\usepackage{graphicx}
\usepackage{amsthm} 
\usepackage{dsfont}
\usepackage{array}
\usepackage{multicol}
\usepackage[utf8]{inputenc}  
\usepackage[frenchb,english]{babel}
\usepackage{MnSymbol}
\usepackage{enumerate} 
\usepackage[blocks]{authblk}
\usepackage{mathtools}

%% file: macros.tex
\theoremstyle{plain}
\newtheorem{defi}{Definition}

\newtheorem{lem}[defi]{Lemma}
\newtheorem{coro}[defi]{Corollary}
\newtheorem{prop}[defi]{Proposition}
\newtheorem{theo}[defi]{Theorem}
\theoremstyle{definition}

\theoremstyle{remark}
\newtheorem{remark}[defi]{Remark}

\numberwithin{defi}{section}

\newcommand{\Pcal}{\mathcal{P}}

\newcommand{\Ccal}{\mathcal{C}}

\newcommand{\Prob}{\mathbb{P}}